\documentclass[12pt]{article}
\usepackage[a4paper,left=3cm, right=3cm,top=3.5cm, bottom=4.25cm]{geometry}
\usepackage{amsmath,amssymb,amsfonts,amsthm}
\usepackage{epsfig}

\let\oldbibliography\thebibliography
\renewcommand{\thebibliography}[1]{%
  \oldbibliography{#1}%
  \setlength{\itemsep}{2pt}%
}

\newtheorem{theorem}{Theorem}[section]

\newtheorem{proposition}[theorem]{Proposition}

\newtheorem*{MainA}{Main Theorem A}
\newtheorem*{MainB}{Main Theorem B}

\def\pfsp{\hskip 1em}

\newcommand{\pff}[2]{\noindent{\it Proof~of~#1\,:}\pfsp #2 {\hfill $\Box$}}

\def \QD1 {\hfill $\spadesuit$}

\newcommand{\case}[2]{\smallskip {\bf Case #1\/:} {\it #2}}

\newcommand{\DF}[1]{{\bf #1\/}}

\newcommand{\set}[2]{\{#1 \;|\; #2 \}}
\newcommand{\ems}{\varnothing}
\newcommand{\sm}{\setminus}

\newcommand{\De}{\Delta}
\newcommand{\de}{\delta}

\newcommand{\pa}{\chi}
\newcommand{\gp}{\rho}

\newcommand{\cn}{\chi}
\newcommand{\ocn}{\overline{\chi}}

\newcommand{\cDG}{{\cal S}{\cal D}}
\newcommand{\cMG}{{\cal MG}}
\newcommand{\cRG}{{\cal RG}}
\newcommand{\cGD}{{\cal DG}}

\newcommand{\Ga}{\Gamma}
\newcommand{\cG}{{\cal G}}

\newcommand{\cP}{{\cal P}}

\newcommand{\CO}{{\cal CO}}
\newcommand{\CR}{{\rm Cri}}

\newcommand{\f}{\varphi}
\newcommand{\fin}{\varphi^{-1}}

\newcommand{\olm}{\overline{m}}

\newcommand{\nat}{\mathbb{N}}
\newcommand{\nato}{\mathbb{N}_0}

\newcommand{\ganz}{\mathbb{Z}}

\newcommand{\has}{\nabla}
\newcommand{\dis}{\boxplus}
\newcommand{\hal}{\nabla^{\ell}}
\newcommand{\halt}{\nabla^t}
\newcommand{\dil}{\boxplus^{\ell}}
\newcommand{\dit}{\boxplus^t}

\newcommand{\Cre}{\mathrm{Ext}}
\newcommand{\cre}{\mathrm{ext}}

\newcommand{\mdo}[3]{#1 \equiv #2 \, \mathrm{(mod} \; #3 \mathrm{)}}

\newcommand{\PP}  {{\sf P}}
\newcommand{\NP}  {{\sf NP}}

\newcommand{\NPH} {{\NP}-hard}

\newcommand{\jou}[4]{{\em #1} {\bf #2} (#3) #4.}
\def \JCTB {J. Combin. Theory Ser.~B}
\def \DM {Discrete Math.}

\numberwithin{equation}{section}

\begin{document}
\title{\bf Point partition numbers: decomposable and indecomposable critical graphs}

\author{Justus von Postel \qquad Thomas Schweser \qquad  Michael Stiebitz\\[0.5ex]
\small Institute of Mathematics\\[-0.8ex]
\small Technische Universit\"at Ilmenau\\[-0.8ex]
\small D-98684 Ilmenau, Germany\\
\small\tt \{justus.von-postel,thomas.schweser,michael.stiebitz\}@tu-ilmenau.de}
\date{}
\maketitle
\begin{abstract}
Graphs considered in this paper are finite, undirected and loopless, but we allow multiple edges. The point partition number $\pa_t(G)$ is the least integer $k$ for which $G$ admits a coloring with $k$ colors such that each color class induces a $(t-1)$-degenerate subgraph of $G$. So $\pa_1$ is the chromatic number and $\pa_2$ is the point aboricity. The point partition number $\pa_t$ with $t\geq 1$ was introduced by Lick and White. A graph $G$ is called $\pa_t$-critical if every proper subgraph $H$ of $G$ satisfies $\pa_t(H)<\pa_t(G)$. In this paper we prove that if $G$ is a $\pa_t$-critical graph whose order satisfies $|G|\leq 2\pa_t(G)-2$, then $G$ can be obtained from two non-empty disjoint subgraphs $G_1$ and $G_2$ by adding $t$ edges between any pair $u,v$ of vertices with $u\in V(G_1)$ and $v\in V(G_2)$. Based on this result we establish the minimum number of edges possible in a $\pa_t$-critical graph $G$ of order $n$ and with $\pa_t(G)=k$, provided that $n\leq 2k-1$ and $t$ is even. For $t=1$ the corresponding two results were obtained in 1963 by Tibor Gallai.
\end{abstract}



\section{Introduction}
Coloring theory for graphs plays a central role in discrete mathematics and has attracted a lot of attention over the past decades. However, coloring theory mainly focuses on the investigation of the chromatic number $\chi$. In studying the chromatic number $\chi$-critical graphs became an important tool since coloring problems for $\chi$ can very often be reduced to problems about $\chi$-critical graphs.  A graph $G$ is \DF{$\chi$-critical} if $\chi(H) < \chi(G)$ for each proper subgraph $H$ of $G$. The class of $\chi$-critical graphs was introduced and investigated by G. A. Dirac in the 1950s (see e.g. \cite{Dirac52,Dirac53,Dirac57,Dirac64}), and the topic of $\chi$-critical graphs has received much attention within the last six decades. In 1963 Gallai \cite{Gallai63a} and \cite{Gallai63b} published two fundamental papers related to the structure of $\chi$-critical graphs. In this paper he proves - among many other results - the following two remarkable theorems.

\begin{theorem} [Gallai 1963]
\label{theorem:gallai1}
Let $G$ be a $\chi$-critical graph of order $n$ and with $\chi(G)=k$ for $k\geq 2$. If $n\leq 2k-2$, then $G$ is obtained from the disjoint union of two non-empty subgraphs $G_1$ and $G_2$ of $G$ by joining each vertex of $G_1$ to each vertex of $G_2$ by exactly one edge.
\end{theorem}

\begin{theorem} [Gallai 1963]
\label{theorem:gallai2}
Let $n$ and $k$ be integers with $n=k+p$ and $2\leq p \leq k-1$. If $\cre(k,n)$ is the minimum number of edges in a $\chi$-critical graph having order $n$ and chromatic number $k$, then $\cre(k,n)={n \choose 2} - (p^2+1).$
\end{theorem}

Our main aim is to extend those two results to the point partition number introduced in 1970 by Lick and White \cite{LickW1970}. In what follows let $t$ be a positive integer. A graph $G$ is called \DF{strictly $t$-degenerate} if every non-empty subgraph $H$ of $G$ has a vertex $v$ whose degree in $H$ satisfies $d_H(v)\leq t-1$. The \DF{point partition number} $\pa_t(G)$ of a graph $G$ is the least non-negative integer $k$ for which $G$ has a coloring with a set of $k$ colors such that each color class induces a strictly $t$-degenerate subgraph of $G$. Note that $\pa_1=\chi$ and $\pa_2$ is referred to as the \DF{point aboricity}. A graph $G$ is called $\pa_t$-critical if every proper subgraph $H$ of $G$ satisfies $\pa_t(H)<\pa_t(G)$. Note that the graphs considered here may have parallel edges. In this paper we shall prove the following two results.

\begin{MainA}
Let $G$ be a $\pa_t$-critical graph of order $n$ and with $\pa_t(G)=k$ for $k\geq 2$. If $n\leq 2k-2$, then $G$ is obtained from the disjoint union of two non-empty subgraphs of $G$, say $G_1$ and $G_2$, by joining each vertex of $G_1$ to each vertex of $G_2$ by exactly $t$ parallel edges.
\end{MainA}

\begin{MainB}
Let $n$ and $k$ be integers with $n=k+p$ and $1\leq p \leq k-1$. If $\cre_t(k,n)$ is the minimum number of edges in a $\pa_t$-critical graph having order $n$ and $t$-chromatic number $k$, then $\cre_t(k,n)=t{n \choose 2}-\tfrac{t}{2}(2p+1)p$, provided that $t$ is even.
\end{MainB}

The rest of the paper is organized as follows. The second section gives a brief introduction to terminology for graphs. In the third section, we establish basic properties of $\pa_t$-critical graphs. In the fourth section, we introduce two fundamental constructions for critical graphs, the Haj\'os join and the Dirac join. In the fifth section we give some background information about Gallai's decomposition result (Theorem~\ref{theorem:gallai1}) and show that our first main result is a simple consequence of a decomposition result for hypergraphs. Our first main result is used in the sixth section to describe the structure of $\pa_t$-critical graphs whose order is near to $\pa_t$. The proof of the second main result is given in the seventh section. The last section contains some concluding remarks.

\section{Preliminaries}
We use the standard notation. In particular, $\nat$ denotes the set of all positive integers and $\nato=\nat \cup \{0\}$. For integers $k$ and $\ell$, let $[k,\ell]=\set{x\in \ganz}{k \leq x \leq \ell}$. In this paper, the term \DF{graph} refers to a finite undirected graph with multiple edges and without loops. For a graph $G$, we denote by $V(G)$ and $E(G)$ the \DF{vertex set} and the \DF{edge set} of $G$, respectively. The number of vertices of $G$ is called the \DF{order} of $G$ and is denoted by $|G|$. A graph $G$ is called \DF{empty} if $|G|=0$, in this case we also write $G=\ems$. For a vertex $v\in V(G)$ let $E_G(v)$ denote the set of edges of $G$ incident with $v$. Then $d_G(v)=|E_G(v)|$  is the \DF{degree} of $v$ in $G$. As usual, $\de(G)=\min_{v\in V(G)}d_G(v)$ is the \DF{minimum degree} and $\De(G)=\max_{v\in V(G)}d_G(v)$ is the \DF{maximum degree} of $G$. For two different vertices $u, v$ of $G$, let $E_G(u,v)=E_G(u)\cap E_G(v)$ be the set of edges between $u$ and $v$. If $e\in E_G(u,v)$ then we also say that $e$ is an edge of $G$ \DF{joining} $u$ and $v$. Furthermore, $\mu_G(u,v)=|E_G(u,v)|$ is the \DF{multiplicity} of the vertex pair $u,v$; and $\mu(G)=\max_{u\not=v}\mu_G(u,v)$ is the \DF{maximum multiplicity} of $G$. The graph $G$ is said to be \DF{simple} if $\mu(G)\leq 1$. For $X,Y \subseteq V(G)$, denote by $E_G(X,Y)$ the set of all edges of $G$ joining a vertex of $X$ with a vertex of $Y$, and put $E_G(X) = E_G(X,X)$. If $G'$ is a \DF{subgraph} of $G$, we write $G'\subseteq G$. The \DF{subgraph} of $G$ \DF{induced by} the vertex set $X$ with $X\subseteq V(G)$ is denoted by $G[X]$, i.e., $V(G[X])=X$ and $E(G[X])=E_G(X)$. Furthermore, $G-X=G[V(G)\sm X]$. For a vertex $v$, let $G-v=G-\{v\}$. For $F\subseteq E(G)$, let $G-F$ denote the subgraph of $G$ with vertex set $V(G)$ and edge set $E(G)\sm F$. For an edge $e\in E(G)$, let $G-e=G-\{e\}$. We denote by $C_n$ the \DF{cycle} of order $n$ with $n\geq 2$, and by $K_n$ the \DF{complete graph} of order $n$ with $n\geq 0$.

In what follows let $t$ be a given positive integer. If $G$ is a graph, then $H=tG$ denotes the graph obtained from $G$ by replacing each edge $e$ of $G$ by $t$ parallel edges with the same two ends as $e$, that is, $V(H)=V(G)$ and for any two different vertices $u,v\in V(G)$ we have $\mu_H(u,v)=t\mu_G(u,v)$. The graph $H=tG$ is called a \DF{$t$-uniform inflation} of $G$. A graph $G$ is said to be \DF{strictly $t$-degenerate} if every non-empty subgraph $G'$ of $G$ has a vertex $v$ such that $d_{G'}(v)\leq t-1$. Let $\cDG_t$ denote the class of strictly $t$-degenerate graphs. Note that $\cDG_1$ is the class of edgeless graphs and $\cDG_2$ is the class of forests.

Let $G$ be a graph and let $\Ga$ be a set. A \DF{coloring} of $G$ with \DF{color set} $\Ga$ is a mapping $\f:V(G) \to \Ga$ that assigns to each vertex $v\in V(G)$ a \DF{color} $\f(v)\in \Ga$. For a color $c\in \Ga$, the preimage $\fin(c)=\set{v\in V(G)}{\f(v)=c}$ is called a \DF{color class} of $G$ with respect to $\f$. A subgraph $H$ of $G$ is called \DF{monochromatic} with respect to $\f$ if $V(H)$ is a subset of a color class of $G$ with respect to $\f$. A coloring $\f$ of $G$ with color set $\Ga$ is called an \DF{$\cDG_t$-coloring} of $G$ if for each color $c\in \Ga$ the subgraph of $G$ induced by the color class $\fin(c)$ belongs to $\cDG_t$. We denote by $\CO_t(G,k)$ the set of $\cDG_t$-colorings of $G$ with color set $\Ga=[1,k]$. The \DF{point partition number} $\pa_t(G)$ of the graph $G$ is defined as the least integer $k$ such that $\CO_t(G,k)\not=\ems$.

The graph classes $\cDG_t$ and the coloring parameters $\pa_t$ with $t\geq 1$ were first introduced and investigated in 1970 by Lick and White \cite{LickW1970}. Bollob\'as and Manvel \cite{BollobasM79} used the term \DF{$t$-chromatic number} for $\pa_t$. Note that $\pa_1$ equals the chromatic number $\chi$, and the parameter $\pa_2$ is also referred to as the \DF{point aboricity}. The point aboricity was introduced in 1968 by Hedetniemi \cite{Hedetniemi68}. Note that any $\cDG_t$-coloring of a graph induces a $\cDG_t$-coloring with the same color set of each of its subgraphs. Consequently, $\pa_t$ is a monotone graph parameter, that is, $G'\subseteq G$ implies $\pa_t(G') \leq \pa_t(G)$.
%
Furthermore, the components of a graph can be colored independently, so if $G\not= \ems$, then
\begin{equation}
\label{Chap2:Equ:chi-components}
\pa_t(G)=\max \set{\pa_t(G')}{G' \mbox{ is a component of } G}.
\end{equation}
Recall that a \DF{block} of a non-empty graph $G$ is a maximal connected subgraph $H$ of $G$ such that $H$ has no separating vertex. If we have an optimal $\cDG_t$-coloring for each block of $G$ and $t\in \{1,2\}$, then we can combine these colorings to obtain an optimal $\cDG_t$-coloring of $G$ by permuting colors in the blocks if necessary. So for every non-empty graph $G$ we have
\begin{equation}
\label{Chap2:Equ:Block}
\pa_t(G)=\max\set{\pa_t(H)}{H \mbox{ is a block of } G} \mbox{ provided that } t\leq 2.
\end{equation}
However, for $t\geq 3$ this is not true in general. Fig.~\ref{Chap1:Fig:Block} shows a graph $G$ with two isomorphic blocks $H_1$ and $H_2$ such that $\pa_3(G)=2$, but $\pa_3(H_1)=\pa_3(H_2)=1$.

\begin{figure}[htbp]
\centering

\includegraphics[height=2cm]{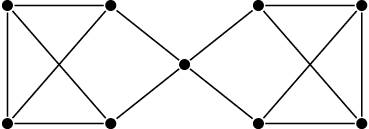}

\caption{A graph $G$ with $\pa_3(G)=2$.}
\label{Chap1:Fig:Block}   
\end{figure}

Clearly, $\pa_t(G)=0$ if and only if $G=\ems$, and $\pa_t(G)\leq 1$ if and only if $G$ belongs to $\cDG_t$. It is well known that a graph belongs to $\cDG_t$ if and only if deleting step by step vertices whose actual degree is at most $t-1$ results in the empty graph. So the decision problem whether a graph $G$ satisfies $\pa_t(G)\leq 1$ belongs to the complexity class {\PP}. In general, however, the determination of the parameter $\pa_t$ is {\NPH}.

\section{The class of $\pa_t$-critical graphs}\label{section_construction}

In studying the $t$-chromatic number $\pa_t$, critical graphs are a useful concept. A graph $G$ is called \DF{$\pa_t$-critical} if $\pa_t(G')<\pa_t(G)$ for every proper subgraph $G'$ of $G$. To see why critical graphs are of great interest, let us consider a \DF{graph property} $\cP$, that is, a class of graphs closed under taking isomorphic copies. Suppose that $\cP$ is \DF{monotone} in the sense that $G' \subseteq G \in \cP$ implies $G'\in \cP$. Furthermore, consider a \DF{graph parameter} $\gp$ defined for the class $\cP$, that is, a mapping that assigns to each graph of $\cP$ a real number such that $\gp(G')=\gp(G)$ whenever $G'$ and $G$ are isomorphic graphs belonging to $\cP$. If we want to bound the $t$-chromatic number for the graphs of $\cP$ from above by the parameter $\gp$, then we can apply the critical graph method, provided that $\gp$ is \DF{monotone}, that is, $G'\subseteq G \in \cP$ implies $\gp(G') \leq \gp(G)$. The proof of the following proposition is easy and left to the reader.

\begin{proposition} Let $t\geq 1$ be a fixed integer, let $\cP$ be a monotone graph property and let $\gp$ be a monotone graph parameter defined for $\cP$. Then the following statements hold:
\begin{itemize}
\item[{\rm (a)}] For every graph $G\in \cP$ there exists a $\pa_t$-critical graph $H\in \cP$ such that $H\subseteq G$ and $\pa_t(H)=\pa_t(G)$.
\item[{\rm (b)}] If $\pa_t(H)\leq \gp(H)$ for every $\pa_t$-critical graph $H\in \cP$, then $\pa_t(G)\leq \gp(G)$ for every graph $G\in \cP$.
\end{itemize}
\label{Chap3:Prop:critical-method}
\end{proposition}

The $t$-chromatic number is a monotone graph parameter and it is easy to check that if we delete a vertex or an edge from a graph, then the $t$-chromatic number decreases by at most one.
As a consequence we obtain the following result.

\begin{proposition} Let $G$ be a graph and let $k\in \nat$. Then $\pa_t(G)\leq k-1$ if and only if there is no $\pa_t$-critical graph $H$ with $H\subseteq G$ and $\pa_t(H)=k$.
\label{Chap3:Prop:coloring=critcal=forbidden}
\end{proposition}
\begin{proof} If $G$ contains a $\pa_t$-critical graph $G'$ with $\pa_t(G')=k$ as a subgraph, then $\pa_t(G)\geq \pa_t(G')=k$. Conversely, if $\pa_t(G)\geq k$, then it follows from the above remark that there is a subgraph $G'$ of $G$ with $\pa_t(G')=k$. By Lemma~\ref{Chap3:Prop:critical-method}(a), $G'$ and hence $G$ contains a $\pa_t$-critical graph $H$ with $\pa_t(H)=k$ as a subgraph.
\end{proof}

The following two propositions list some basic properties of $\pa_t$-critical graphs; the proofs are straightforward and left to the reader.

\begin{proposition} Let $G$ be a $\pa_t$-critical graph with $\pa_t(G)=k$ and $k \geq 1$. Then the following statements hold:
\begin{itemize}
\item[{\rm (a)}] If $v\in V(G)$ and $\f\in \CO_t(G-v,k-1)$, then $|E_G(v,\fin(c))|\geq t$ for every color $c\in [1,k-1]$.
\item[{\rm (b)}] If $e\in E_G(u,v)$ and $\f \in \CO_t(G-e,k-1)$, then $\f(u)=\f(v)$ and $\mu_{G-e}(u,v)\leq t-1$.
\item[{\rm (c)}] $\de(G)\geq t(k-1)$ and $\mu(G)\leq t$.

\item[{\rm (d)}] $|G|\geq k$ and equality holds if and only if $G=tK_k$.
\item[{\rm (e)}] $G$ is connected, and if $t\leq 2$, then $G$ has no separating vertex.
\end{itemize}
\label{Chap3:Prop:coloring-property}
\end{proposition}

\begin{proposition} Let $G$ be a simple graph. Then $\pa_t(tG)=\chi(G)$ and, moreover, $tG$ is $\pa_t$-critical if and only if $G$ is $\chi$-critical.
\label{Chap3:Prop:chi(G)=pat(tG)}
\end{proposition}

For integers $k, n\in \nato$, let $\CR_t(k)$ denote the class of $\pa_t$-critical graphs $G$ with $\pa_t(G)=k$ and let
$$\CR_t(k,n)=\set{G\in \CR_t(k)}{|G|=n}.$$
Since a graph $G$ satisfies $\pa_t(G)=0$ if and only if $G=\ems$, and $\pa_t(G)\geq 1$ if and only if $V(G)\not=\ems$, it follows from Proposition~\ref{Chap3:Prop:coloring=critcal=forbidden} that
$\CR_t(0)=\{\ems\}$ and $\CR_t(1)=\{K_1\}$. Let $\cRG_t$ denote the class of connected $t$-regular graphs. Then it is easy to check that $\cRG_t\subseteq \CR_t(2)$ and
$$\CR_2(2)=\cRG_2=\set{C_n}{n\geq 2}.$$
K\"onig's characterization of bipartite graphs implies that
$$\CR_1(3)=\set{C_n}{\mdo{n}{1}{2}}.$$
For any fixed $k\geq 3$, a good characterization of the class $\CR_t(k)$ seems to be unlikely.

While the class of $\chi$-critical graphs has attracted a lot of attention, this is not the case for  the class of $\pa_t$-critical graphs with $t\geq 2$. The papers by Kronk and Mitchen \cite{KronkM75}, by Bollob\'as and Harary \cite{BollobasH74}, by Mihok \cite{Mihok81}, and by \v{S}krekovski \cite{Skrekovski2002} are all devoted to the structure of $\pa_2$-critical simple graphs.  The two papers by Thomason \cite{Thomason79} \cite{Thomason82} deal with $\pa_t$-critical simple graphs for $t\geq 1$, and
the papers by Schweser \cite{Schweser2019} and by Schweser and Stiebitz \cite{SchweserS2019} contain some results about $\pa_t$-critical graphs having multiple edges and with arbitrary $t\geq 1$.

If we want to check whether a given graph is $\pa_t$-critical, it suffices to investigate all edge deleted graphs. This follows from the following trivial result.

\begin{proposition} Let $G$ be a graph and let $k\geq 2$ be an integer. Then $G\in \CR_t(k)$ if and only if $\de(G)\geq 1$ and $\pa_t(G- e)<k \leq \pa_t(G)$ for every edge $e\in E(G)$.
\label{Chap3:Prop:critcal=edgedeleted}
\end{proposition}

Let $G$ be a $\pa_t$-critical graph with $\pa_t(G)=k$ and $k\geq 1$. By Proposition~\ref{Chap3:Prop:coloring-property}(c), $\de(G)\geq t(k-1)$, which
leads to a natural way of dividing the vertices of $G$ into two classes.
The vertices of $G$
having degree $t(k-1)$ are called \DF{low vertices}
of $G$, and the remaining vertices are called \DF{high vertices} of $G$.
So any high vertex of $G$ has degree at least $t(k-1)+1$ in $G$. Furthermore, the subgraph of $G$ induced by its low vertices is called the \DF{low vertex subgraph} of $G$. For $\chi$-critical graphs, this classification is due to Gallai \cite{Gallai63a}.
The following two results due to Schweser \cite[Theorems 3 and 5]{Schweser2019} (with $\cP=\cDG_t$ and $r=t$) generalizes Gallai's theorem about the structure of the low vertex subgraph of $\chi$-critical graphs and gives a Brooks-type result for $\pa_t$.

\begin{theorem} [Schweser 2019]
Let $G$ be a $\pa_t$-critical graph with $\pa_t(G)=k$ and $k\geq 1$, and let $B$ be a block of the low vertex subgraph of $G$. Then $B=sK_n$ with $1\leq s\leq t$ and $n\geq 1$, or $B=sC_n$ with $1\leq s\leq t$ and $n\geq 3$ odd, or $B$ is a connected $t$-regular graph, or $B\in \cDG_t$ and $\De(B)\leq t$.
\label{Chap3:Theo:Lowvertexsubgraph}
\end{theorem}

\begin{theorem}[Schweser 2019]
If $G$ is a connected graph, then $\pa_t(G)\leq \lceil \De(G)/t\rceil+1$ and equality holds only if $G=sK_{(t/s)p+1}$ for some integers $s,p$ with $1\leq s \leq t, p \geq 0, \mdo{t}{0}{s}$, or $G=tC_n$ for $n \geq 3$ odd, or $G$ is a connected $t$-regular graph.
\label{Chap3:Theo:Brooks}
\end{theorem}

It is easy to show that in the above theorem we can replace "only if" by "if and only if". As a consequence the graphs listed in the above theorem are the only $\pa_t$-critical graphs having no high vertices. For the class of simple graphs Theorem~\ref{Chap3:Theo:Brooks} was obtained for $t=1$ by Brooks \cite{Brooks41} and for $t\geq 2$ by Bollob\'as and Manvel \cite{BollobasM79} as well as by Borodin \cite{Borodin76}.

\smallskip

A graph is \DF{$\pa_t$-vertex-critical} if $\pa_t(G')<\pa_t(G)$ for every proper induced subgraph $G'$ of $G$. Clearly, every $\pa_t$-critical graph is $\pa_t$-vertex-critical, but not conversely. Examples of $\chi$-vertex-critical graphs that are not $\chi$-critical were given by Dirac. Obviously, a graph $G$ is $\pa_t$-vertex-critical if and only if $\pa_t(G-v)<\pa_t(G)$ for every vertex $v\in V(G)$. Results about critical graphs can be often transformed into results about the larger class of vertex-critical graphs.

\begin{proposition} Let $G$ be a $\pa_t$-vertex-critical graph. Then $G$ contains a $\pa_t$-critical subgraph $G'$ with $\pa_t(G')=\pa_t(G)$ and any such subgraph has the same vertex set as $G$.
\label{Chap3:Prop:critcal=vertex-critical}
\end{proposition}
\begin{proof} That $G$ contains a $\pa_t$-critical subgraph $G'$ with $\pa_t(G')=\pa_t(G)$ follows from Proposition~\ref{Chap3:Prop:critical-method}(a). Now let $G'$ be any such a subgraph. If a vertex $v$ of $G$ does not belong to $G'$, then $G'\subseteq G-v$ and, since $G$ is $\pa_t$-vertex-critical, we obtain that $\pa_t(G')\leq \pa_t(G-v)<\pa_t(G)$, which is impossible.
\end{proof}

\section{Constructions for $\pa_t$-critical graphs}

There are two well known constructions for $\chi$-critical graphs that can easily be extended to $\pa_t$-critical graphs, namely the Dirac join and the Haj\'os join. The first construction is very common in graph theory and was first used by Dirac (see Gallai \cite[(2.1)]{Gallai63a}) to construct $\cn$-critical graphs, and the second construction was invented by Haj\'os \cite{Hajos61} to characterize the class of graphs with chromatic number at least $k$.

In this section let $\ell$ be a given positive integer. Let $G_1$ and $G_2$ be two \DF{disjoint graphs}, that is, $G_1$ and $G_2$ have no vertex and no edge in common. Let $G$ be the graph obtained from the union $G_1 \cup G_2$ by adding edges between $V(G_1)$ and $V(G_2)$ so that $\mu_G(u,v)=\ell$ whenever $u\in V(G_1)$ and $v\in V(G_2)$. We call $G$ the \DF{Dirac $\ell$-join} of $G_1$ and $G_2$ and write $G=G_1 \dil G_2$. The proof of the following theorem is easy and left to the reader.

\begin{theorem}[Dirac Construction]
Let $G=G_1 \dit G_2$ be the Dirac $t$-join of two disjoint non-empty graphs $G_1$ and $G_2$. Then $\pa_t(G)=\pa_t(G_1)+\pa_t(G_2)$ and $G$ is $\pa_t$-critical if and only if both $G_1$ and $G_2$ are $\pa_t$-critical.
\label{Chap4:Theo:Dirac-join}
\end{theorem}

Let $G_1$ and $G_2$ be two disjoint graphs and, for $i\in \{1,2\}$, let $(u_i,v_i)$ be a pair of distinct vertices of $G_i$ and let $E_i\subseteq E_{G_i}(u_i,v_i)$ be a set of $\ell$ edges.
Let $G$ be the graph obtained from the union $G_1 \cup G_2$ by deleting the  edge sets $E_1$  and $E_2$ from $G_1$ and $G_2$, respectively, identifying the vertices $v_1$ and $v_2$, and adding $\ell$ new edges between $u_1$ and $u_2$. We then say that $G$ is the \DF{Haj\'os $\ell$-join} of $G_1$ and $G_2$ and write $G=(G_1,u_1,v_1,E_1)\hal (G_2,u_2,v_2,E_2)$, or briefly $G=G_1 \hal G_2$. The Haj\'os 1-join is also called the \DF{Haj\'os join} and we write $\has$ rather than $\has^1$. Figure~\ref{Chap4:Fig:K4hasK4} shows the Haj\'os joins $G=2K_4\has 2K_4$ and $G'=2K_4\has^2 2K_4$

\begin{figure}[htbp]
\centering
\includegraphics[height=1.7cm]{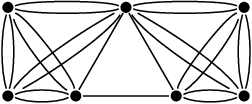}
\hspace{1cm}
\includegraphics[height=1.7cm]{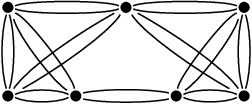}
%
%
\caption{Two Haj\'os joins of two $2K_4$'s.}
\label{Chap4:Fig:K4hasK4}   
\end{figure}

\begin{proposition}
Let $G=G_1\halt G_2$ be the Haj\'os $t$-join of two disjoint non-empty graphs $G_1$ and $G_2$. Then $\pa_t(G) \geq \min\{\pa_t(G_1), \pa_t(G_2)\}.$
\label{Chap4:Prop:minimalcolor}
\end{proposition}
\begin{proof}
Suppose that $G=(G_1,u_1,v_1,E_1)\halt (G_2,u_2,v_2,E_2)$ and denote by $v^*$ the vertex of $G$ obtained by identifying $v_1$ and $v_2$. Assume that $\pa_t(G) = k$. Then there is a coloring $\f \in \CO_t(G,k)$. Because $\mu_G(u_1,u_2)=t$ we obtain that $\f(u_1) \neq \f(u_2)$ and, therefore, $\f(v^*) \neq \f(u_i)$ for an $i \in \{1,2\}$, say $i=1$. Then the restriction of $\varphi$ to $G_1$ induces a coloring $\f_1\in \CO_t(G_1,k)$ and so $\pa_t(G_1)\leq k=\pa_t(G)$.
\end{proof}

A graph $G$ is called \DF{Haj\'os-$k$-constructible} if $G$ is a simple graph that can be obtained from disjoint copies of $K_k$ by repeated application of the Haj\'os join and the identification of two non-adjacent vertices. Haj\'os \cite{Hajos61} proved that any graph with chromatic number at least $k$ contains a Haj\'os-$k$-constructible subgraph. Urquhart \cite{Urquhart97} extended Haj\'os' result and proved that, for $k\geq 3$, a simple graph $G$ has chromatic number at least $k$ if and only if $G$ itself is Haj\'os-$k$-constructible. Dirac and, independently, Gallai proved that if $G_1$ and $G_2$ are two disjoint simple graphs and $k\geq 4$, then $G_1 \has G_2$ belongs to $\CR_1(k)$ if and only if both $G_1$ and $G_2$ belong to $\CR_1(k)$ (for a proof see also the paper by Schweser, Stiebitz and Toft \cite[Theorem 9]{SchweserST2019}). Note that $C_7=C_4\has C_4$ belongs to $\CR_1(3)$, but $C_4$ does not belong to $\CR_1(3)$. The next theorem shows that the result of Dirac and Gallai has a counterpart for the point aboricity $\pa_2$.

\begin{theorem}[Haj\'os Construction I] \label{theorem_Hajos-join-I}
Let $G=G_1\has G_2$ be the Haj\'os join of two disjoint non-empty graphs $G_1$ and $G_2$. Then the following statements hold:
\begin{itemize}
	\item[\upshape (a)] $\pa_2(G) \geq \min\{\pa_2(G_1), \pa_2(G_2)\}$.
	\item[\upshape (b)] If $\pa_2(G_1)=\pa_t(G_2)=k$ and $k \geq 2$, then $\pa_t(G)=k$.
	\item[\upshape (c)] If both $G_1$ and $G_2$ belong to $\CR_t(k)$ and $k \geq 2$, then $G$ belongs to $\CR_t(k)$.
	\item[\upshape (d)] If $G$ belongs to $\CR_t(k)$ and $k \geq 2$, then both $G_1$ and $G_2$
	belong to $\CR_t(k)$.
\end{itemize}
\label{Chap4:Theo:Hajos-join}
\end{theorem}
\begin{proof}
Let $G=(G_1,u_1,v_1,\{e_1\})\has (G_2,u_2,v_2,\{e_2\})$. Denote by $v^*$ the vertex of $G$ obtained from the identification of $v_1$ and $v_2$, and let $e^*\in E_G(u_1,u_2)$ be the new edge.

In order to proof (a) let $\pa_2(G) = k$. Then there is a coloring $\f \in \CO_2(G,k)$. For $i\in \{1,2\}$, $\f$ induces a coloring $\f_i\in \CO_2(G_i-e_i,k)$ with $\f_i(v_i)=\f(v^*)$. We claim that either $\f_1 \in \CO_2(G_1,k)$ or $\f_2 \in \CO_2(G_2,k)$. Otherwise, for $i\in \{1,2\}$, there is a monochromatic cycle $C_i$ of $G_i$ with respect to $\f_i$ that contains the edge $e_i$.
But then the Haj\'os join $C=(C_1,u_1,v_1,\{e_1\})\has (C_2,u_2,v_2,\{e_2\})$ is a monochromatic cycle of $G$ with respect to $\f$, which is impossible.

For the proof of (b) let $\pa_2(G_1) = \pa_2(G_2)=k$ with $k \geq 2$. Then, for $i\in \{1,2\}$, there is a coloring $\f_i \in \CO_2(G_i,k)$. By permuting colors if necessary, we may choose these two colorings so that $\f_1(v_1)=\f_2(v_2)$. Let $\f:V(G)\to [1,k]$ be the mapping induced by $\f_1\cup \f_2$, that is, $\f(v)=\f_i(v)$ if $v\in V(G_i-v_i)$ and $\f(v^*)=\f_1(v_1)=\f_2(v_2)$. We claim that $\f \in \CO_2(G,k)$. For otherwise, there is a monochromatic cycle $C$ with respect to $\f$, and this cycle contains the vertex $v^*$ and the edge $e^*$. Using the edge $e_i$, we can then construct a monochromatic cycle $C_i$ with respect to $\f_i$, a contradiction. Hence $\f \in \CO_2(G,k)$ as claimed. Then $\pa_2(G)\leq k$ and (a) implies that $\pa_t(G)=k$.

In order to prove (c) suppose that both $G_1$ and $G_2$ belong to $\CR_2(k)$ and $k \geq 2$. Then $\de(G)\geq 1$ and $\pa_2(G)=k$ (by (b)). Hence in order to show that $G\in \CR_2(k)$ it suffices to show that $\pa_2(G-e)\leq k-1$ for all edges $e\in E(G)$ (by Proposition~\ref{Chap3:Prop:critcal=edgedeleted}).

\case{1}{$e=e^*$.} Since $G_i$ belongs to $\CR_2(k)$, there exists a coloring $\f_i\in CO_2(G_i-e_i,k-1)$ and $\f_i(u_i)=\f_i(v_i)$ (by Proposition~\ref{Chap3:Prop:coloring-property}(b)). By permuting colors if necessary, we may assume that $\f_1(v_1)=\f_2(v_2)$. Then $\f_1\cup \f_2$ induces a coloring $\f\in CO_2(G-e^*,k-1)$ and so $\pa_2(G-e)\leq k-1$.

\case{2}{$e \not= e^*$.} Then $e$ belongs to $G_i-e_i$ for some $i\in \{1,2\}$, and by symmetry, we may assume that $i=1$. Since $G_1\in \CR_2(k)$, there is a coloring $\f_1 \in \CO_2(G_1 - e,k-1)$. Since $G_2 \in \CR_2(k)$, there exists also a coloring $\f_2 \in \CO_2(G_2 - e_2,k-1)$ and $\f_2(u_2)=\f_2(v_2)$ (by Proposition~\ref{Chap3:Prop:coloring-property}(b)). By permuting colors if necessary, we may choose $\f_2$ so that $\f_2(v_2)=\f_1(v_1)$. Then $\f_1 \cup \f_2$ induces a coloring $\f\in \CO_2(G-e,k-1)$ and, hence, $\pa_2(G-e)\leq k-1$. For otherwise, there is a monochromatic cycle $C$ with respect to $\f$ and $C$ contains $v^*$ and $e^*$, but neither $e_1$ nor $e_2$. Since $e\not=e_1$ we obtain from $C$ and $e_1$ a monochromatic cycle with respect to $\f_1$, which is impossible. This completes the proof of (c).

For the proof of (d) assume that $G$ belongs to $\CR_2(k)$ and $k \geq 2$. Our aim is to show that both $G_1$ and $G_2$ belong to $\CR_2(k)$. To this end, we apply Proposition~\ref{Chap3:Prop:critcal=edgedeleted}. Since $G\in \CR_2(k)$, $G$ is connected and has no separating vertex (by Proposition~\ref{Chap3:Prop:coloring-property}(e)). Hence $G_i$ is connected and so $\de(G_i)\geq 1$ for $i\in \{1,2\}$. Next we claim that $\pa_2(G_1)\geq k$. For otherwise there is a coloring $\f_1\in \CO_2(G_1,k-1)$. As $G\in \CR_2(k)$, $\pa_2(G-e^*)\leq k-1$, which leads to a coloring $\f_2\in \CO_2(G_2-e_2,k-1)$. By permuting colors if necessary, we may assume that $\f_1(v_1)=\f_2(v_2)$. Let $\f:V(G)\to [1,k-1]$ be the mapping induced by $\f_1 \cup \f_2$ (i.e.,
$\f(v)=\f_i(v)$ if $v\in V(G_i-v_i)$ and $\f(v^*)=\f_1(v_1)=\f_2(v_2)$). If there exists a monochromatic cycle $C$ with respect to $\f$, then both $v^*$ and $e^*$ belong to $C$ and, by using $e_1$, $C$ yields a monochromatic cycle with respect to $\f_1$, which is impossible. Hence $\f\in \CO_2(G,k-1)$, which is impossible. This proves the claim that $\pa_2(G_1)\geq k$. By a similar argument it follows that $\pa_2(G_2)\geq k$.

It remains to show that $\pa_2(G_i-e) < k$ for every $e \in E(G_i)$ and $i\in \{1,2\}$. By symmetry, it suffices to show this for $i=1$. If $e=e_i$, then $G_1-e_1$ is a proper subgraph of $G$ and, since $G\in \CR_2(k)$, we obtain that $\pa_2(G_1-e_1)\leq k-1$. Now assume that $e\not= e_i$. Then there is a coloring $\f\in \CO_2(G-e,k-1)$. Then $\f$ induces a coloring $\f_2\in \CO_2(G_2-e_2,k-1)$. As $\pa_2(G_2)\geq k$, this implies that $\f(v^*)=\f(u_2)$ and there is a monochromatic path $P$ in $G_2-e_2$ with respect to $\f$, whose ends are $v^*$ and $u_2$. Hence there is no monochromatic path in $G_1-e_1$ with respect to $\f$ whose ends are $v^*$ and $u_1$. Consequently, $\f$ induces a coloring $\f_1\in \CO_2(G_1-e)$ and so $\pa_2(G_1-e)\leq k-1$. This completes the proof of (d).
\end{proof}

\begin{theorem}[Haj\'os Construction II] \label{theorem_Hajos-join-II}
Let $G=G_1\has^2 G_2$ be the Haj\'os $2$-join of two disjoint non-empty graphs $G_1$ and $G_2$. Then the following statements hold:
\begin{itemize}
\item[\upshape (a)] If $\pa_2(G_1)=\pa_2(G_2)=k$ and $k \geq 3$, then $\pa_2(G)=k$.
\item[\upshape (b)] If both $G_1$ and $G_2$ belong to $\CR_2(k)$ and $k \geq 3$, then $G$ belongs to $\CR_2(k)$.
\end{itemize}
\end{theorem}
\begin{proof}
Let $G=(G_1,u_1,v_1,E_1)\has^2 (G_2,u_2,v_2,E_2)$. Denote by $v^*$ the vertex of $G$ obtained by identifying $v_1$ and $v_2$, and let $E^*=E_G(u_1,u_2)$ be the set of new edges. Note that $|E_1|=|E_2|=|E^*|=2$.

For the proof of (a), assume that $\pa_2(G_1)=\pa_2(G_2)=k$ and $k\geq 3$. To prove that $\pa_2(G)=k$ it suffices to show that $\pa_2(G)\leq k$ (by Propostion~\ref{Chap4:Prop:minimalcolor}). By assumption there is a coloring $\f_i\in \CO_2(G,k)$ for $i\in \{1,2\}$. As $\mu_{G_i}(u_i,v_i)\geq |E_i|=2$, we obtain that $\f_i(u_i)\not=\f_i(v_i)$. As $k\geq 3$ we can permute colors if necessary such that $\f_1(v_1)=\f_2(v_2)$ and $\f_1(u_1)\not=\f_2(u_2)$. Then, $\f_1\cup \f_2$ induces a coloring $\f\in \CO_2(G,k)$ and so $\pa_2(G)\leq k$ as required.

For the proof of (b) assume that both $G_1$ and $G_2$ belong to $\CR_2(k)$ and $k\geq 3$. Then $\de(G)\geq 1$ and $\pa_2(G)=k$ (by (a)). Hence to prove that $G\in \CR_2(k)$ it suffices to show that $\pa_2(G-e)\leq k-1$ for all edges $e\in E(G)$ (by Proposition~\ref{Chap3:Prop:critcal=edgedeleted}). By Proposition~\ref{Chap3:Prop:coloring-property}(c) it follows that $\mu_{G_i}(u_i,v_i)=2$ for $i\in \{1,2\}$.

\case{1}{$e\in E^*$.} Let $i\in \{1,2\}$ and let $e_i\in E_i$ be an edge. Since $G_i$ belongs to $\CR_2(k)$, there is a coloring $\f_i\in CO_2(G_i-e_i,k-1)$ and $\f_i(u_i)=\f_1(v_i)$ (by Proposition~\ref{Chap3:Prop:coloring-property}(b)). Since $|E_i|=2$, $G_i-E_i$ contains no monochromatic path with respect to $\f_i$ between $u_i$ and $v_1$. By permuting colors we may assume that $\f_1(v_1)=\f_2(v_2)$. Then $\f_1 \cup \f_2$ induces a coloring $\f\in \CO_2(G-e,k-1)$ and so $\pa_2(G-e)\leq k-1$.

\case{2}{$e \not\in E^*$.} By symmetry we may assume that $e \in E(G_1)\sm E_1$. Since $G_1$ belongs to $\CR_2(k)$, there is a coloring $\f_1 \in \CO_2(G_1 - e,k-1)$. Since $e\not\in E_{G_1}(u_1,v_1)$ and $\mu_{G_1}(u_1,v_1)=2$, we obtain that $\f_i(u_1)\not=\f_1(v_1)$. Let $e_2\in E_2$ be an edge. Since $G_2$ belongs to $\CR_2(k)$, there is a coloring $\f_2\in \CO_2(G_2-e_2,k-1)$ and $\f_2(v_2)=\f_2(u_2)$. By permuting colors if necessary, we may assume that $\f_1(v_1)=\f_2(v_2)$.
Then $\f_1\cup \f_2$ induces a coloring $\f\in \CO_2(G-e,k-1)$ and so $\pa_2(G-e)\leq k-1$.
This completes the proof of (b).
\end{proof}

Using the Haj\'os join and the Dirac join, it is well known and easy to show that if $k\geq 4$, then
$$\CR_1(k,k)=\{K_k\} \mbox{ and } \CR_1(k,n)\not=\ems  \mbox{ if and only if } n\geq k \mbox{ and } n\not=k+1.$$
If $G\in \CR_1(k)$, then $tG\in \CR_t(k)$ (by Proposition~\ref{Chap3:Prop:chi(G)=pat(tG)}). Consequently, we have
$$\CR_t(k,k)=\{tK_k\} \mbox{ and } \CR_t(k,n)\not=\ems  \mbox{ if } n\geq k \mbox{ and } n\not=k+1.$$
Clearly, $\CR_t(k,n)=\ems$ if $n<k$ and it is easy to show that $\CR_t(k,k+1)\not=\ems$ if and only if $t\geq 2$ (see also Sect.~\ref{Sect:Critical_near_to_order}). If $t=2s\geq 2$, then $sK_3\in \CR_t(2)$ and hence $tK_{k-2}\dit sK_3\in \CR_t(k,k+1)$ (by Theorem~\ref{Chap4:Theo:Dirac-join}).

\section{Indecomposable $\pa_t$-critical graphs}

In what follows, let $\cMG_t$ denote the class of graphs $G$ satisfying $\mu(G)\leq t$. So $\cMG_1$ is the class of simple graphs. Following Gallai, a graph of $\cMG_t$ is called \DF{$t$-decomposable} if it is the Dirac $t$-join of two non-empty disjoint subgraphs; otherwise the graph is called \DF{$t$-indecomposable}. By Theorem~\ref{Chap4:Theo:Dirac-join} it follows that a $t$-decomposable $\pa_t$-critical graph is the Dirac $t$-join of its $t$-indecomposable $\pa_t$-critical subgraphs. So the $t$-indecomposable $\pa_t$-critical graphs are  building elements of $\pa_t$-critical graphs. In 1963, Gallai \cite{Gallai63b} proved the following remarkable result about indecomposable $\chi$-critical graphs.

\begin{theorem} [Gallai 1963]
If $G$ is a $1$-indecomposable $\chi$-vertex-critical graph, then $|G|\geq 2\chi(G)-1$.
\label{Chap5:Theo:Gallai-vertex-critical}
\end{theorem}

Note that the statement of Theorem~\ref{Chap5:Theo:Gallai-vertex-critical} immediately implies the statement of Theorem~\ref{theorem:gallai1}; that the converse implication also holds, follows from Proposition~\ref{Chap3:Prop:critcal=vertex-critical}.

Let $G$ be a graph belonging to $\cMG_t$. To decide whether the graph $G$ is $t$-decomposable we can use its $t$-complement. We call a graph $H$ the \DF{$t$-complement} of $G$, written $H=\overline{G}^t$, if $V(H)=V(G)$, $E(G)\cap E(H)=\ems$, and $\mu_H(u,v)+\mu_G(u,v)=t$ for every pair $(u,v)$ of distinct vertices of $G$. Note that the $1$-complement corresponds to the ordinary complement of a simple graph. Clearly, $H=\overline{G}^t$ if and only if $G=\overline{H}^t$. If $G$ has order $n$, then $G\cup \overline{G}^t=tK_n$. Furthermore,
\begin{equation}
\label{Chap5:Equ:decomposable=connected}
\mbox{$G$ is $t$-decomposable if and only if $\overline{G}^t$ is disconnected.}
\end{equation}

For a simple graph $G$, the chromatic number of the complement of $G$ is called the \DF{cover number} of $G$, written $\ocn(G)$. Hence, $\ocn(G)$ is the least integer $k$ for which $G$ has a coloring with $k$-colors such that each color class induces a complete graph. A simple graph $G$ is \DF{$\ocn$-vertex-critical} if $\ocn(G-v)<\ocn(G)$ for every vertex $v\in V(G)$. As $\ocn(G)=\cn(\overline{G})$,
$G$ is $\ocn$-vertex-critical if and only if $\overline{G}$ is $\chi$-vertex-critical.
So Theorem~\ref{Chap5:Theo:Gallai-vertex-critical} is equivalent to the following result.

\begin{theorem} [Gallai 1963]
If $G$ is a connected $\ocn$-vertex-critical graph, then $|G|\geq 2\ocn(G)-1$.
\label{Chap5:Theo:Gallai-cover-critical}
\end{theorem}

There are three known proofs of Gallai's result. The original proof given by T. Gallai applies matching theory to $\ocn$-vertex-critical graphs; so he first proved Theorem~\ref{Chap5:Theo:Gallai-cover-critical} and obtained Theorem~\ref{Chap5:Theo:Gallai-vertex-critical} as a corollary. The second proof is due to Molloy \cite{Molloy99}; he applies Berge's version of Tutte's perfect matching theorem to $\ocn$-vertex-critical graphs. A third proof is due to Stehl\'ik \cite{Stehlik03}; his proof also deals with $\ocn$-vertex-critical graphs, but the proof uses no matching theory. Stiebitz and Toft \cite{StiebitzT2016} adapted Stehl\'ik's argument to give a direct proof of Gallai's result from first principles. This proof can be easily extended to $\cn$-critical hypergraphs, see the paper by Stiebitz, Storch, and Toft \cite{StiebitzST2016}. Our first main result is an immediate consequence of the decomposition result for hypergraphs.

\smallskip

A hypergraph $H$ is a pair of sets,
$V(H)$ and $E(H)$, where $V(H)$ is finite and $E(H)$
is a subset of $2^{V(H)}$ such that $|e|\geq 2$ for all $e\in E(H)$.
The set $V(H)$ is the \DF{vertex set} of $H$ and its elements are the \DF{vertices} of $H$. The set $E(H)$ is the \DF{edge set} of $H$ and its elements are the \DF{edges} of $H$. An edge $e$ with $|e|\geq 3$ is called a \DF{hyperedge}, and an edge $e$ with $|e|=2$ is called an \DF{ordinary edge}. So our hypergraphs have no loops and no parallel edges. A hypergraph $H$ is called \DF{simple} if no edge of $H$ is contained in another edge of $H$. A \DF{coloring} of a hypergraph $H$ with \DF{color set} $\Ga$ is a mapping $\f:V(G) \to \Ga$; we call $\f$ a \DF{proper coloring} of $H$ if no edge of $H$ is monochromatic with respect to $\f$, that is, $|\f(e)|\geq 2$ for every edge $e\in E(H)$. The \DF{chromatic number} of $H$, denoted by $\cn(H)$, is the least integer $k$ such that $H$ has a proper coloring with a set of $k$ colors. The hypergraph $H$ is called \DF{$\cn$-critical} if $\cn(H')<\cn(H)$ for every proper subhypergraph $H'$ of $H$; we call $H$ \DF{$\cn$-vertex-critical} if $\cn(H-v)<\cn(H)$ for all vertices $v$ of $H$. Clearly, each $\cn$-critical hypergraph is $\cn$-vertex-critcal, but not conversely. If $H_1$ and $H_2$ are two vertex disjoint hypergraphs, then $H=H_1\dis H_2$ denotes the hypergraph with $V(H)=V(H_1) \cup V(H_2)$ and $E(H)=E(H_1)\cup E(H_2) \cup \set{\{u,v\}}{u\in V(H_1),v\in V(H_2)}$, we the call $H$ the \DF{Dirac join} of $H_1$ and $H_2$. The following result is due to Stiebitz, Storch, and Toft \cite{StiebitzST2016}, and generalizes Gallai's decomposition result.

\begin{theorem}
If $H$ is a $\cn$-critical hypergraph with $|V(H)|\leq 2\cn(H)-2$, then $H$ is the Dirac join of two non-empty hypergraphs.
\label{Chap5:Theo:Gallai-hyper}
\end{theorem}

\begin{theorem}
If $G$ is a $\pa_t$-critical graph, whose $t$-complement is connected, then $|G|\geq 2\pa_t(G)-1$. Equivalently, if $G$ is a $\pa_t$-critical graph with $|G|\leq 2\pa_t(G)-2$, then $G$ is $t$-decomposable.
\label{Chap5:Theo:Gallai-pt}
\end{theorem}
\begin{proof}
For the graph $G$, we construct a hypergraph $H$ as follows. The vertex set is $V(H)=V(G)$, and a subset $e\subseteq V(G)$ is an edge of $H$ if and only if $G[e]\not\in \cDG_t$, but $G[e]-v\in \cDG_t$ for every $v\in V(G)$. Clearly, $H$ is a simple hypergraph. Furthermore, $\f:V(G)\to \Ga$ is a map, then $\f$ is an $\cDG_t$-coloring of $G$ if and only if $\f$ is a proper coloring of $H$. Consequently $\cn(H)=\pa_t(G)$. Since $G$ is $\pa_t$-critical, it follows that $H$ is $\cn$-vertex-critical. By deleting edges of $H$, we obtain a $\cn$-critical hypergraph $H'$ with $\cn(H')=\cn(H)$ and $V(H')=V(H)$. hence $|V(H')|=|G|\leq 2\pa_t(G)-2=2\cn(H')-2$. Then Theorem~\ref{Chap5:Theo:Gallai-hyper} implies that $H'$ is the Dirac join of two non-empty hypergraphs. Since $H$ is a simple hypergraph and $V(H)=V(H')$, the same holds for $H$. Clearly, if $e$ is an ordinary edge of $H$, then $G[e]=tK_2$. Hence $G$ is the Dirac $t$-join of two non-empty graphs. This proves the theorem.
\end{proof}

\section{Critical graphs whose order is near to $\pa_t$}
\label{Sect:Critical_near_to_order}

Let $G$ be a graph belonging to $\cMG_t$. A non-empty subgraph $H$ of $G$ is called \DF{$t$-dominating}, if there is a non-empty subgraph $G'$ such that $G=H\dit G'$. Clearly, any $t$-dominating subgraph of $G$ is an induced subgraph of $G$. Suppose that $G$ is $\pa_t$-critical. Then any $t$-dominating subgraph of $G$ is $\pa_t$-critical, too (by Theorem~\ref{Chap4:Theo:Dirac-join}). If $H$ is a $t$-dominating subgraph of $G$ with $\pa_t(H)=1$, then $H=K_1$.
For $t\geq 2$, let $K_3(t)$ denote the graph $\tfrac{t}{2}K_3$ if $t$ is even, and $\tfrac{t+1}{2}K_3$ minus an edge if $t$ is odd. Then it is easy to check that $\CR_t(2,1)=\ems$, $\CR_t(2,2)=\{tK_2\}$, $\CR_1(2,n)=\ems$ provided that $n\geq 3$, and $\CR_t(2,3)=\{K_3(t)\}$ for $t\geq 2$. Note that $\CR_t(0)=\{\ems \}$. We shall apply Theorem~\ref{Chap5:Theo:Gallai-pt} to deduce the following result. The case $t=1$ of this result was obtained by Gallai \cite{Gallai63b}.

\begin{theorem}
Let $G$ be a $\pa_t$-critical graph with $\pa_t(G)=k$ and $k\geq 1$, let $p$ be the number of $t$-dominating subgraphs of $G$ belonging to $\CR_t(1)$, and let $q$ be the number of $t$-dominating subgraphs of $G$ belonging to $\CR_t(2)$ and having order at least $3$. Then the following statements hold:
\begin{itemize}
\item[{\rm (a)}] $0\leq p \leq k$ and there exists a graph $G'\in \CR_t(k-p)$ such that $G=tK_p\dit G'$, $G'$ has no $t$-dominating subgraph isomorphic to $K_1$, and $|G'| \geq \tfrac{3}{2}(k-p).$ Furthermore, $p\geq 3k-2|G|$ and equality holds if and only if $t\geq 2$ and $G'$ is the Dirac $t$-join of $\tfrac{1}{2}(k-p)$ disjoint subgraphs of $G$ each of which is isomorphic to $K_3(t)$.
\item[{\rm (b)}] $0 \leq p+2q\leq k$ and there exists a graph $G_1\in \CR_t(2q)$ and a graph $G_2\in \CR_t(k-p-2q)$
such that $H=tK_p \dit G_1 \dit G_2$, $G_1$ is the Dirac sum of $q$ graphs each of which belongs to $\CR_t(2)$ and has order at least 3, $G_2$ has no $t$-dominating subgraph belonging to $\CR_t(1) \cup \CR_t(2)$, and $|G_2|\geq \tfrac{5}{3}(k-p-2q)$. Furthermore, $2p+q\geq 5k-3|G|$ and equality holds if and only if $t\geq 2$, $G_1$ is the Dirac $t$-join of $q$ disjoint subgraphs of $G$ each of which is isomorphic to $K_3(t)$, and $G_2$ is the Dirac $t$-join of $\tfrac{1}{3}(k-p-2q)$ disjoint subgraphs of $G$ each of which belongs to $\CR_t(3,5)$.
\end{itemize}
\label{Chap7:Theo:Gallai-pt-zwei}
\end{theorem}
\begin{proof}
In what follows, let $G$ be an arbitrary graph belonging to $\CR_t(k)$ with $k\geq 1$. Then $G$ is a connected graph of $\cMG_t$ (by Proposition~\ref{Chap3:Prop:coloring-property}(c)(e)) and hence
$$G=G_1 \dit G_2 \dit \cdots  \dit G_s,$$
where $\overline{G}_1^t, \overline{G}_2^t, \ldots, \overline{G}_s^t$ are the components of $\overline{G}^t$. For $i\in \{1,2, \ldots, s\}$, let $k_i=\pa_t(G_i)$ and $n_i=|G_i|$. By Theorem~\ref{Chap4:Theo:Dirac-join}, we obtain that
\begin{itemize}
\item[{\rm (1)}] $k=k_1+k_2 + \cdots + k_t$ and $G_i\in \CR_t(k_i,n_i)$ for $i\in [1,s]$.
\end{itemize}
Since $\overline{G}_i^t$ is connected, Theorem~\ref{Chap5:Theo:Gallai-pt} implies that
\begin{itemize}
\item[{\rm (2)}] $|G_i|\geq 2k_i-1$ for $i\in [1,s]$.
\end{itemize}
Since $\CR_t(1)=\{K_1\}$, $\CR_t(2,2)=\{tK_2\}$, $\CR_1(2,3)=\ems$, and $\CR_t(2,3)=\{K_3(t)\}$ for $t\geq 2$, we obtain that $k_i=1$ and $G_i=K_1$, or $k_i=2$ and $|G_i|\geq 3$ (where equality holds if and only if $t\geq 2$ and $G_i=K_3(t)$), or $k_i\geq 3$ and $|G_i|\geq 5$. For a subset $I$ of $[1,s]$, let $G_I=\dit_{i\in I}G_i$ be the Dirac $t$-join of the graphs $G_i$ with $i\in I$, and let $k_I=\sum_{i\in I}k_i,$ where $G_{\ems}=\ems$ and $k_{\ems}=0$. By Theorem~\ref{Chap4:Theo:Dirac-join}, $G_I\in \CR_t(k_I)$. Let $P=\set{i\in [1,s]}{k_i=1}$, $Q=\set{i\in [1,s]}{k_i=2}$, $R=[1,s]\sm (P \cup Q)$, $p=|P|$, $q=|Q|$, and $r=|R|$. Then $P,Q$ and $R$ are pairwise disjoint sets whose union is $[1,s]$. Thus we obtain that
\begin{itemize}
\item[{\rm (3)}] $G=G_P\dit G_Q \dit G_R$, where $G_P=tK_p$ and $G_Q\in \CR_t(2q)$.
\end{itemize}
Note that $p$ is the number of $t$-dominating subgraphs of $G$ belonging to $\CR_t(1)$, and $q$ is the number of $t$-dominating subgraphs of $G$ belonging to $\CR_t(2)$ and having order at least $3$. In particular, $q=0$ if $t=1$.

First let us establish a lower bound for $p$. So let $\overline{P}=[1,s] \sm P$. Then $\overline{P}=R \cup Q$ and $G=G_P\dit G_{\overline{P}}$.
For $i\in \overline{P}$, we have that $k_i\geq 2$ and so, by (2), $|G_i|\geq 2k_i-1\geq \tfrac{3}{2}k_i$, where equality holds if and only if $t\geq 2$ and $G_i=K_3(t)$. By Theorem~\ref{Chap4:Theo:Dirac-join} and (1), we conclude that $k_P=p$ and $k_{ \overline{P}}=k-p$. For the order of $G$, it then follows from (1) and (2) that
$$|G|=p+\sum_{i\in \overline{P}}|G_i|\geq p+\tfrac{3}{2}\sum_{i\in \overline{P}}k_i= p+\tfrac{3}{2}(k-p),$$
which is equivalent to $p\geq 3k-2|G|$. Clearly, $p= 3k-2|G|$ if and only if $t\geq 2$ and $G_{\overline{P}}$ is the Dirac $t$-join of $\tfrac{1}{2}(k-p)$ disjoint $K_3(t)$'s. This proves (a).

For $i\in R$, we have $k_i\geq 3$ and so, by (2), $|G_i|\geq 2k_i-1\geq \tfrac{5}{3}k_i$, where equality holds if and only if $G_i\in \CR_t(3,5)$.
By Theorem~\ref{Chap4:Theo:Dirac-join}, we have $k_P=p$, $k_Q=2q$, and $k_R=k-p-2q$.
For the order of $G$ we then obtain that
$$|G|=p+\sum_{i\in Q}|G_i|+\sum_{i\in R}|G_i|\geq p+3q+\tfrac{5}{3}\sum_{i\in R}k_i=p+3q+\tfrac{5}{3}(k-p-2q),$$
which is equivalent to $2p+q\geq 5k-3|G|$. Clearly, $2p+q= 5k-3|G|$ if and only if $t\geq 2$, $G_i=K_3(t)$ for all $i\in Q$ and $G_i\in \CR_t(3,5)$ for all $i\in R$. Thus (b) is proved.
\end{proof}

For a graph $K\in \cMG_t$ and a class of graphs $\cG\subseteq \cMG_t$, define the class $K\dit \cG$ by $K\dit \cG=\set{K\dit G}{G\in \cG}$ if $\cG\not= \ems$, and $K \dit \cG=\ems$ otherwise. If $\cG$ is a graph property, then we do not distinguish between isomorphic graphs, so we are only interested in the number of isomorphism types of $\cG$, that is, the number of equivalence classes of $\cG$ with respect to the isomorphism relation for graphs.

The number of isomorphism types of the class $\CR_t(k,n)$ is finite, where $\CR_t(k,n)=\ems$ if $n<k$ and $\CR_t(k,k)=\{tK_k\}$. Furthermore, $\CR_t(1,n)=\ems$ if $n>1$, $\CR_t(2,3)=\{K_3(t)\}$ if $t\geq 2$ and $\CR_1(2,3)=\ems$.

From Theorem~\ref{Chap7:Theo:Gallai-pt-zwei}(a) we conclude that $\CR_t(k,k+1)=K_1\dit \CR(k-1,k)$ if $k\geq 3$, which implies by induction on $k$ that if $k\geq 2$, then
\begin{equation}
\label{Chap7:Equ:CReven(k,k+1)}
\CR_t(k,k+1)=\{tK_{k-2}\dit K_3(t)\} \mbox{ if $t\geq 2$, and } \CR_1(k,k+1)=\ems.
\end{equation}
For the rest of this section, assume that $t\geq 2$. For the class $\CR_t(4,6)$ we then conclude from Theorem~\ref{Chap7:Theo:Gallai-pt-zwei}(b) that
$$\CR_t(4,6)=(K_1\dit \CR_t(3,5)) \cup \{K_3(t) \dit K_3(t)\}.$$
By Theorem~\ref{Chap7:Theo:Gallai-pt-zwei}(a), it follows that $\CR(k,k+2)=K_1\dit \CR(k-1,k+1)$ if $k\geq 5$, which implies by induction on $k$ that
$$\CR(k,k+2)=(tK_{k-4}\dit K_3(t) \dit K_3(t)) \cup (tK_{k-3}\dit \CR(3,5))$$
if $k\geq 4$. If $n=k+3$, then we conclude from Theorem~\ref{Chap7:Theo:Gallai-pt-zwei}(b) that
\begin{eqnarray*}
\CR_t(5,8)&=&(K_1\dit \CR_t(4,7)) \cup (K_3(t)\dit\CR_t(3,5)) \cup \CR',
\end{eqnarray*}
where $\CR'=tK_2\dit \CR_t(3,6)$,
and from Theorem~\ref{Chap7:Theo:Gallai-pt-zwei}(a) we get
$$\CR_t(6,9)=(K_1\dit \CR_t(5,8))\cup \{K_3(t) \dit K_3(t) \dit K_3(t)\}.$$
If $k\geq 7$, then Theorem~\ref{Chap7:Theo:Gallai-pt-zwei}(a) implies that $$\CR(k,k+3)=K_1 \dit \CR_t(k-1,k+2).$$ By induction on $k$, we then obtain that if $k\geq 6$ , then
\begin{eqnarray*}
\CR_t(k,k+3)&=&(tK_{k-4} \dit \CR_t(4,7)) \cup (tK_{k-5}\dit K_3(t) \dit \CR_t(3,5)) \cup \CR'',
\end{eqnarray*}
where
$$\CR''=tK_{k-6}\dit (K_1\dit \CR_t(5,8) \cup \{K_3(t) \dit K_3(t) \dit K_3(t) \}).$$

%
%

\section{Critical graphs with few edges}

In this section we shall investigate the extremal function $\cre_t(\cdot,\cdot)$ defined by
$$\cre_t(k,n)=\min \set{|E(G)|}{G\in \CR_t(k,n)}$$
and the corresponding class of extremal graphs defined by
$$\Cre_t(k,n)=\set{G\in \CR_t(k,n)}{|E(G)|=\cre_t(k,n)},$$
where $k$ and $n$ are positive integers. From Proposition~\ref{Chap3:Prop:coloring-property}(c) it follows that
\begin{equation}
\label{Chap8:Equ:trivial}
\cre_t(k,n)\geq \tfrac{1}{2}t(k-1)n,
\end{equation}
and Theorem~\ref{Chap3:Theo:Lowvertexsubgraph} tells us when equality holds. The function $\cre_1(k,n)$ is well investigated, a survey about the many partial results obtained in this case can be found in the paper by Kostochka \cite{Kostochka2006}. That it is worthwhile to study the function $\cre_1(k,n)$ was first emphasized by Dirac \cite{Dirac57} and subsequently by Gallai \cite{Gallai63a,Gallai63b} and by Ore \cite{Ore67}. In 2014, Kostochka and Yancey \cite{{KostY14}} succeeded in determining the best linear approximation of the function $\cre_1(k,n)$.

\begin{theorem} [Kostochka and Yancey 2014]
If $n\geq k\geq 4$ and $n\not=k+1$, then
$$\cre_1(k,n)\geq \frac{(k+1)(k-2)n-k(k-3)}{2(k-1)},$$
where equality holds if $\mdo{n}{1}{k-1}$. As a consequence, we have that
$$\lim_{n\to \infty}\frac{\cre_1(k,n)}{n}=\frac{1}{2}(k-\frac{2}{k-1})$$
\label{Chap8:Theo:KostYan}
\end{theorem}

For the function $\cre_t(k,n)$ with $t\geq 2$ only two improvements of the trivial lower bound \eqref{Chap8:Equ:trivial} are known. Both improvements are due to \v{S}krekovski \cite{Skrekovski2002}, but he only considers simple graphs and $t=2$.

Based on Theorem~\ref{Chap5:Theo:Gallai-vertex-critical}, Gallai \cite{Gallai63b} established the exact
values for the function $\cre_1(k,n)$ including a description of the extremal classes $\Cre_1(k,n)$, provided that $k+2 \leq n \leq 2k-1$. For $k\geq 3$, let ${\cGD}(k)$ be the class of simple graphs $G$ whose vertex set consists  of three non-empty pairwise disjoint sets $X, Y_1$ and $Y_2$ with
    $$|Y_1|+|Y_2|=|X|+1=k-1$$
and two additional vertices $v_1$ and $v_2$ such that $G[X]$ and $G[Y_1 \cup Y_2]$ are complete graphs not joined by any edge in $G$, and $N_G(v_i)=X \cup Y_i$ for $i\in \{1,2\}$. Then it is easy to show that $\cGD(k)\subseteq \CR_1(k,2k-1)$. The class $\cGD(k)$ was discovered by Dirac \cite{Dirac74} and by Gallai \cite{Gallai63a}. Note that all graphs belonging to $\cGD(k)$ are $1$-indecomposable.

\begin{theorem} [Gallai 1963]
Let $n=k+p$ be an integer, where $k,p\in \nat$ and $2\leq p \leq k-1$.
Then $$\cre_1(k,n)={n \choose 2} - (p^2+1)=\frac{1}{2}((k-1)n+p(k-p)-2)$$
and $\Cre_1(k,n)=K_{k-p-1}\dis^1{\cGD}(p+1)$.
\label{Chap8:Theorem:Gallai2}
\end{theorem}

Based on Theorem~\ref{Chap5:Theo:Gallai-pt} we shall prove a counterpart of Theorem~\ref{Chap8:Theorem:Gallai2}, but only when $t$ is even.

\begin{theorem}
Let $n=k+p$ be an integer, where $k,p\in \nat$ and $1\leq p \leq k-1$, and let $t$ be an even positive integer.
Then $$\cre_t(k,n)=t{n \choose 2} - \tfrac{t}{2}(2p+1)p=\tfrac{t}{2}(k^2-k+2kp-p^2-2p)$$
and $\Cre_t(k,k+p)=\{tK_{k-p-1}\dit \tfrac{t}{2}K_{2p+1}\}$.
\label{Chap8:Theorem:Gallai3}
\end{theorem}

For the proof of the above theorem, the following result is useful. For a graph $G$, let $e(G)$ denote the number of edges of $G$.

\begin{theorem}
\label{Chap8:Theorem:half}
Let $t$ be a positive integer, and let $G\in \CR_t(k,n)$ be a graph with $n=k+p$ and $2\leq p \leq k-2$. If $G$ has no $t$-dominating subgraph belonging to $\CR_t(1) \cup \CR_t(2)$, then $e(G)\geq t{n \choose 2}-tp^2$.
\end{theorem}
\begin{proof}
Let $G\in \CR_t(k,n)$ be a graph with $n=k+p$ and $2\leq p \leq k-2$ such that $G$ has no $t$-dominating subgraph belonging to $\CR_t(1) \cup \CR_t(2)$. Note that $G\in \cMG_t$ and so $e(G)+e(\overline{G}^t)=t{n \choose 2}$. Our aim is to show that $e(G)\geq t{n \choose 2}-tp^2$, which is equivalent to $e(\overline{G}^t)\leq tp^2$. Since $n\leq 2k-2$ it follows from Theorem~\ref{Chap5:Theo:Gallai-pt} that $G$ is $t$-decomposable. Hence
$$G=G_1 \dit G_2 \dit \cdots  \dit G_s,$$
where $\overline{G}_1^t, \overline{G}_2^t, \ldots, \overline{G}_s^t$ are the components of $\overline{G}^t$, and $s\geq 2$. For $i\in [1,s]$, let $k_i=\pa_t(G_i)$, $n_i=|G_i|$, $m_i=e(G_i)$ and $\olm_i=e(\overline{G}_i^t)=t{n_i \choose 2}-m_i$. By Theorem~\ref{Chap4:Theo:Dirac-join}, we obtain that
$$k=k_1+k_2 + \cdots + k_s \mbox{ and } G_i\in \CR_t(k_i,n_i) \mbox{ for } i\in [1,s].$$
Since $G$ has no $t$-dominating subgraph belonging to $\CR_t(1) \cup \CR_t(2)$, $k_i\geq 3$ for $i\in [1,s]$. As $\overline{G_i}^t$ is connected, we obtain that $n_i\geq 2k_i-1\geq 5$ for $i\in [1,s]$ (by Theorem~\ref{Chap5:Theo:Gallai-pt}). For a subset $I$ of $[1,s]$, let
$$G_I=\dit_{i\in I}G_i, k_I=\sum_{i\in I}k_i, n_I=\sum _{i\in I}n_i \mbox{ and } \olm_I=\sum _{i\in I}\olm_i,$$
where the sum over the empty set is zero. By Theorem~\ref{Chap4:Theo:Dirac-join}, $G_I\in \CR_t(k_I,n_I)$. Our aim is to show
$$e(\overline{G}^t)=\olm_{[1,s]}\leq tp^2.$$
To this end, we divide the set $[1,s]$ into two disjoint subsets, namely $A=\set{i\in [1,s]}{ n_i=2k_i-1}$ and $B=\set{i\in [1,s]}{n_i\geq 2k_i}$. Let $a=|A|$ and $b=|B|$. Since $A \cap B=\ems$ and $A\cup B=[1,s]$, we obtain that
\begin{itemize}
\item[{\rm (1)}] $G=G_A\dit G_B$ and $a+b=s\geq 2$.
\end{itemize}
From (1) and Theorem~\ref{Chap4:Theo:Dirac-join} it follows
\begin{itemize}
\item[{\rm (2)}] $k=k_A+k_B$ and $n=n_A+n_B$.
\end{itemize}
The definition of $A$ implies that
\begin{itemize}
\item[{\rm (3)}] $n_i=2k_i-1, k_i\geq 3$ and $n_i\geq 5$ whenever $i\in A$,
\end{itemize}
from which we obtain that
\begin{itemize}
\item[{\rm (4)}] $k_A=\sum_{i\in A}k_i\geq 3a$ and $n_A=\sum_{i\in A}n_i=2k_A-a$.
\end{itemize}
Since $G_i\in \CR_t(k_i,n_i)$, we conclude that $2m_i\geq t(k_i-1)n_i$ (by Proposition~\ref{Chap3:Prop:coloring-property}(c)). Since $n_i=2k_i-1$ for $i\in A$ (by (3)), this leads to
\begin{itemize}
\item[{\rm (5)}] $2\olm_i=2 t{n_i \choose 2} -2m_i\leq tn_i(n_i-k_i)=t{n_i \choose 2}$ whenever $i\in A$.
\end{itemize}
By (3) and (4), this implies that
$$2\olm_A=\sum_{i\in A} 2\olm_i\leq t\sum_{i\in A}{n_i \choose 2}\leq t((a-1){5 \choose 2}+{n_A-5(a-1) \choose 2}),$$
which is equivalent to
\begin{itemize}
\item[{\rm (6)}] $2\olm_A \leq t(2(k_A-3a)^2+9(k_A-3a)+10a)$.

\end{itemize}
Note that this inequality also holds if $a=0$. If $b\geq 1$, then $G_B\in \CR_t(k_B,n_B)$, where
\begin{itemize}
\item[{\rm (7)}] $k_B\geq 3$
\end{itemize}
and
\begin{itemize}
\item[{\rm (8)}] $n_B\geq 2k_B$.
\end{itemize}
By Proposition~\ref{Chap3:Prop:coloring-property}(c), it follows that $2e(G_B)\geq t(k_B-1)n_B$.
Let $$\olm(G_B)=t{n_B \choose 2} -e(G_B) \mbox{ and } \olm=\olm_A+\olm(G_B).$$ Then
\begin{itemize}
\item[{\rm (9)}] $2\olm(G_B)\leq tn_B(n_B-k_B)=t((n_B-k_B)^2+k_B(n_B-k_B))$.
\end{itemize}
Note that (8) and (9) also hold if $b=0$. Using (6) and (9), we obtain that
\begin{itemize}
\item[{\rm (10)}] $2\olm\leq
    t(2(k_A-3a)^2+9(k_A-3a)+10a+(n_B-k_B)^2+k_B(n_B-k_B))$.
\end{itemize}
Clearly, it suffices to show that
\begin{itemize}
\item[{\rm (11)}] $\olm\leq tp^2$.
\end{itemize}
Using (2) and (4), we obtain that
$$p=n-k=n_A-k_A+n_B-k_B=k_A-a+n_B-k_B,$$
which yields
\begin{eqnarray*}
2p^2 &=& 2((k_A-3a)+2a)^2+2(n_B-k_B)^2+4(k_A-a)(n_B-k_B)\\
&=& 2(k_A-3a)^2+8a(k_A-3a)+8a^2+2(n_B-k_B)^2+\\
&& 4(k_A-a)(n_B-k_B).
\end{eqnarray*}
Together with (10), this leads to
\begin{eqnarray*}
(12) \quad 2(tp^2-\olm)&\geq& t((8a-9)(k_A-3a)+a(8a-10)+\\
&&(n_B-k_B)((n_B-2k_B)+4(k_A-a))).
\end{eqnarray*}
If $a\geq 2$, then (11) follows from (12), (4) and (8). If $a=1$ and $b\geq 1$, then $n_B-k_B\geq k_B\geq 3$ (by (7) and (8)). From (12), (4) and (8) we then conclude that
$$2(tp^2-\olm)\geq (-k_A+1+4k_B(k_A-1)>0.$$
If $a=0$ and $b\geq 2$, then $k_A=0$ and (11) follows from (12) and (8). Since $a+b=s\geq 2$, this shows that (11) holds.  This completes the proof of the theorem.
\end{proof}

\pff{Theorem~\ref{Chap8:Theorem:Gallai3}}{
Let $t$ be an even positive integer, and let $G\in \CR_t(k,n)$ be a graph with $n=k+p$ and $1 \leq p \leq k-1$. Furthermore, let $e(\cdot,\cdot)$ be the function defined by
$$e(k,p)=t{n \choose 2} - \tfrac{t}{2}(2p+1)p=\frac{t}{2}(k^2-k+2kp-p^2-2p).$$
Note that $e(k,p)$ is an integer. Our aim is to show that $e(G)\geq e(k,p)$ and that equality holds if and only if
$G=tK_{k-p-1}\dit \tfrac{t}{2}K_{2p+1}$. The proof is by induction on $k$. If $k=2$, then $p=1$ and so $G\in \CR_t(2,3)=\{\tfrac{t}{2}K_3\}$ (by \eqref{Chap7:Equ:CReven(k,k+1)} and $t$ even). Consequently, $e(G)=3t/2=e(2,1)$ and $G=\tfrac{t}{2}K_3$. This proves the basic case.

Now assume that $k\geq 3$. If $p=1$, then $G=tK_{k-2}\dit \tfrac{t}{2}K_3$ (by \eqref{Chap7:Equ:CReven(k,k+1)})
and $2e(G)=2e(k,1)=t(k^2+k-3)$, and we are done. If $p=k-1$, then $n=2k-1$ and $2e(k,k-1)=t(k-1)(2k-1)$. Since $G\in \CR_t(k,n)$, it follows from Proposition~\ref{Chap3:Prop:coloring-property}(c) that $\de(G)\geq t(k-1)$ and hence $2e(G)\geq t(k-1)n=t(k-1)(2k-1)=2e(k,k-1)$. If $e(G)=e(k,k-1)$, then $\De(G)=\de(G)=t(k-1)$ and $\pa_t(G)=k$, which implies, by Theorem~\ref{Chap3:Theo:Brooks}, that $G=\tfrac{t}{2}K_{2k-1}$. Hence we are done, too. It remains to consider the case when $2\leq p \leq k-2$. Note that this implies, in particular, that $k\geq 4$. Furthermore, Theorem~\ref{Chap5:Theo:Gallai-pt} implies that $G$ is $t$-decomposable.

\case{1}{$G$ has a $t$-dominating subgraph belonging to $\CR_t(1)$.} Then $G=K_1\dit G'$ and $G'\in \CR_t(k-1,n')$ (by Theorem~\ref{Chap4:Theo:Dirac-join}) with $n'=n-1=k+p-1$. Clearly, $e(G)=e(G')+tn'=e(G')+t(k+p-1)$. Furthermore, it is easy to check that $e(k-1,p)+t(k+p-1)=e(k,p)$.
From the induction hypothesis it follows that
$$e(G)=e(G')+t(k+p-1)\geq e(k-1,p)+t(k+p-1)=e(k,p).$$
Furthermore, $e(G)=e(k,p)$ is equivalent to $e(G')=e(k-1,p)$, which is equivalent to $G'=tK_{k-p-2}\dit \frac{t}{2}K_{2p+1}$ and, therefore, to $G=tK_{k-p-1}\dit \frac{t}{2}K_{2p+1}$. This settles the first case.

\case{2}{$G$ has a $t$-dominating subgraph belonging to $\CR_t(2)$, but no such graph belonging to $\CR_t(1)$.} Then, by Theorem~\ref{Chap4:Theo:Dirac-join}, $G=H\dit G'$ with $H\in \CR_t(2,q)$, $G'\in \CR_t(k-2,n')$, where $q\geq 3$ and $n'=k+p-q$. Since $G$ has no $t$-dominating $K_1$ as a subgraph, $n'\geq k-1$ and so $p-q\geq -1$. Since $p\leq k-2$ and $q\geq 3$, we have $n'\leq 2k-5$. Hence the induction hypothesis implies that $e(G')\geq e(k-2,p-q+2)$. Since $H\in \CR_t(2,q)$, it follows that $\de(H)\geq t$ and so $2e(H)\geq tq$. Since $G=H\dit G'$, we obtain that
$$2e(G)=2e(H)+2e(G')+2tn'q\geq 2e(G')+tq(2(k+p-q)+1).$$
This leads to
\begin{eqnarray*}
2e(G)-2e(k,p) &\geq& 2e(k-2,p-q+2)-2e(k,p)+tq(2(k+p-q)+1)\\
              &=& t(-10-8p+11q+4pq-3q^2)\\
              &=& t(q-2)(3(p-q)+p+5)\geq t
\end{eqnarray*}
as $q\geq 3$, $p\geq 2$, and $p-q\geq -1$. Hence $e(G)>e(k,p)$ and we are done.

\case{3}{$G$ has no $t$-dominating subgraph belonging to $\CR_t(1) \cup \CR_t(2)$.} Then Theorem~\ref{Chap8:Theorem:half} implies that $e(G)\geq t{n \choose 2}-tp^2 \geq e(k,p)+1$ and we are done. This completes the proof of the theorem.
}

\section{Concluding remarks}

Let $t\in \nat$. Then $2\cre_t(k,n)\geq t(k-1)n$ (by Proposition~\ref{Chap3:Prop:coloring-property}(c)). If $t$ is even, then $\tfrac{t}{2}K_{2k-1}\in \CR_t(k,2k-1)$, which implies that $2\cre_t(k,2k-1)= t(k-1)(2k-1)$ and $\Cre_t(k,2k-1)=\{\tfrac{t}{2}K_{2k-1}\}$ (by Theorem~\ref{Chap3:Theo:Brooks}). This is the key observation for proving Theorem~\ref{Chap8:Theorem:Gallai3}. If $t$ is odd, it seems very likely that 
$$\Cre_t(k,k+p)=tK_{k-p-1}\dit \Cre_t(p+1,2p+1).$$ 
So it would be helpful to establish $\cre_t(k,2k-1)$ and $\Cre_t(k,2k-1)$. By Gallai's result (Theorem~\ref{Chap8:Theorem:Gallai2}), we know that $\Cre_1(k,2k-1)=\cGD(k)$. Let $\cGD_t(k)=\set{tG}{G\in \cGD(k)}$. Clearly, $\cGD_t(k)\subseteq \CR_t(k,2k-1)$ (by Proposition~\ref{Chap3:Prop:chi(G)=pat(tG)}). However, we do not know whether for odd $t$ we have $\cGD_t(k)\subseteq \Cre_t(k,2k-1)$. Note that $\cre_t(k,n)\leq t\cdot \cre_1(k,n)$ for all $n\geq k$ with $n\not=k+1$. It would be also interesting to investigate the function
$$\cre_t(k,n,m)=\min \set{e(G)}{G\in \CR_t(k,n)\cap \cMG_m}.$$
Clearly, $\cre_t(k,n,t)=\cre_t(k,n)$, $\cre_t(k,n,m)\geq \cre_t(k,n,m+1)$, and, moreover,  $\cre_t(k,n,t)\leq t\cdot \cre_1(k,n,1)$. As pointed out by Kostochka, Schweser, and Stiebitz \cite{KostochkaSS2019}, if $t\geq 1$, $k\geq 3$ and $n>kt+1$, then
$$2\cre_t(k+1,n,1)\geq \left( kt+\frac{kt-2}{(kt+1)^2-3}\right)n+\frac{2kt}{(kt+1)^2-3}.$$
The bound follows from Theorem~\ref{Chap3:Theo:Lowvertexsubgraph}. For $t=1$, this bound was established by Gallai \cite{Gallai63a}, and for $t=2$ the bound was established by \v{S}krekovski \cite{Skrekovski2002}. Note that $\Cre_t(k+1,tk+1)=\{K_{tk+1}\}$ (by Theorem~\ref{Chap3:Theo:Brooks}). G. A. Dirac obtained another bound for $\cre_1(k,n)$, which is, for small values of $n$, better than the Gallai bound. In 1957, Dirac~\cite{Dirac57} proved that every graph $G\in \CR_1(k,n)$
with $n\geq k+2\geq 5$ satisfies
$$2e(G)\geq (k-1)n+k-2$$
and in 1974, he proved in \cite{Dirac74} that equality holds if and only if $G\in \cGD(k)$. \v{S}krekovski \cite{Skrekovski2002} established a Dirac type bound for $t=2$, he proved that
$$2\cre_2(k,n,1)\geq 2(k-1)n+2(k-2) \mbox{ provided that } n\geq 2k\geq 6.$$
As pointed out by \v{S}krekovski \cite{Skrekovski2002}, $\cGD(2k-1)\subseteq \CR_2(k)$ and so $2\cre_2(k,n,1)=2(k-1)n+2(k-2)$ if $n=4k-3$.

In what follows let $k\geq 2$ be a fixed integer. Note that $\CR_2(k,n)\not= \ems$ if and only if $n\geq k$. If $G_1\in \Cre_2(k,n)$ and $G_2\in \Cre_2(k,n')$, then $G=G_1\has G_2$ belongs to $\CR_2(k,n+n'-1)$ (by Theorem~\ref{Chap4:Theo:Hajos-join}) and $e(G)=e(G_1)+e(G_2)-1$. As a consequence we obtain that
\begin{equation}
\label{Chap8:Equ:ext(k,n+m)}
\cre_2(k,n+n')\leq \cre_2(k,n)+\cre_2(k,n'+1)-1.
\end{equation}
By Fekete's lemma, this implies that that there is a constant $L_k$ such that
$$\lim_{n\to \infty} \frac{\cre_2(k,n)}{n}=L_k.$$
We have $\CR_2(2)=\set{C_n}{n\geq 2}$, which implies that $L_2=1$. Since $2C_{2p+1}\in \CR_2(3)$, we have $L_3=2$. So let $k\geq 4$. As $\Cre_2(k,2k-1)=\{K_{2k-1}\}$, we have $\cre_2(k,2k-1)=(2k-1)(k-1)$. By \eqref{Chap8:Equ:ext(k,n+m)}, this leads to
\begin{equation}
\label{Chap8:Equ:ext(k,n+m)N}
\cre_2(k,n+2k-2)\leq \cre_2(k,n)+(2k-1)(k-1)-1.
\end{equation}
It would be useful to know whether we have equality in \eqref{Chap8:Equ:ext(k,n+m)N}. If equality holds, this would lead to
$$L_k=\frac{2k-1}{2}-\frac{1}{2k-2}.$$
If this is true, then we have $\cre_2(k,n)=\cre_2(k,n,1)$ provided that $k\geq 4$. As shown in section four, the Haj\'os join well behaves with respect to the point aboricity $\pa_2$. The Haj\'os join not only preserves the point aboricity, but also criticality. However, this is not the case with respect to the $t$-chromatic number $\pa_t$ for $t\geq 3$.

\end{document}